\documentclass{amsart}
\usepackage{amssymb,amsthm, latexsym, amsfonts, amscd, amsthm, enumerate, nicefrac,easybmat,caption}

\numberwithin{equation}{section}

\newtheorem{theorem}{Theorem}[section]
\newtheorem{conjecture}[theorem]{Conjecture}
\newtheorem{proposition}[theorem]{Proposition}
\newtheorem{corollary}[theorem]{Corollary}

\newtheorem{lemma}[theorem]{Lemma}

\newtheorem{definition}[theorem]{Definition}
\newtheorem{remark}[theorem]{Remark}

\newcommand{\optimal}{{extremal }}
\newcommand{\Optimal}{{Extremal }}
\def\ell{{l}}

\def\H{H}

\def\ul{\underline{\lambda}}

\def\Z{\mathbb{Z}}    \def\Q{\mathbb{Q}}

       \def\cF{{\mathcal F}}

\def\ul{\underline{\lambda}}

\begin{document}
\markboth{Julian Rosen}
{Multiple Harmonic Sums and Wolstenholme's Theorem}
\title{MULTIPLE HARMONIC SUMS AND WOLSTENHOLME'S THEOREM}

\author{JULIAN ROSEN}

\address{Department of Mathematics, University of Michigan\\
530 Church Street\\
Ann Arbor, MI 48109, USA\\ 
}

\date{January 31, 2013}

\maketitle


\begin{abstract}
We give a family of  congruences for the binomial coefficients ${kp-1\choose p-1}$ in terms of multiple harmonic sums, a generalization of the harmonic numbers. Each congruence in 
this family (which depends on an additional parameter $n$)
involves a linear combination of $n$ multiple harmonic sums, and holds $\mod{p^{2n+3}}$.  
The coefficients in these 
congruences are integers depending on $n$ and $k$, but independent of $p$. More generally, 
we construct a  family of congruences $\cF_{2n,k}$ for ${kp-1\choose p-1} \mod{p^{2n+3}}$, 
whose members contain a variable number of terms, and 
show that in this family there is a unique ``optimized'' congruence involving the fewest terms. The special case $k=2$ and $n=0$
recovers  Wolstenholme's theorem   ${2p-1\choose p-1}\equiv 1\mod{p^3}$, valid for all primes $p\geq 5$.
We also characterize those triples $(n, k, p)$ 
for which the optimized congruence holds modulo an extra power of $p$: they are precisely those with either $p$ dividing the numerator of the Bernoulli number $B_{p-2n-k}$, or $k \equiv 0, 1 \mod p$.
\end{abstract}

%

%
%
\section{Introduction}
In 1862 the Rev.\ J.\ Wolstenhome \cite{Wol1862}
noted the congruence that for all primes $p \geq 5$,
\[ 
{2p-1\choose p-1}\equiv 1 \pmod{p^3}.
\]
This result is now called Wolstenholme's
theorem.   Later it was found that the related congruence
on harmonic numbers $H_{n} :=\sum_{j=1}^n\frac{1}{j}$, 
stating that for all primes $p \ge 5$,
\[
H_{p-1}  \equiv 0\pmod{p^2},
\]
which was discovered
 earlier (by E. Waring \cite{War1782} 
in 1782 and again by C.\ Babbage \cite{Bab1819} in 1819), is in fact equivalent to  
Wolstenholme's result.

In the following 150 years, Wolstenholme's
 congruence has been generalized
in many directions (see  Me\v strovi\' c \cite{Mes112} for a survey).
This paper considers  generalizations in  two directions.
The first direction treats a larger set of binomial
coefficients, replacing $2p-1$ with $kp-1$.  In 1900 Glaisher \cite{Gla1900c}
showed that for all integers $k\geq 2$,
\begin{equation}
\label{glaisher}
{ kp-1 \choose p-1} \equiv 1 \pmod{p^3}
\end{equation}
holds for all $p \geq 5$. 
In 1999 Andrews \cite{And99} extended Glaisher's 
theorem to $q$-binomial coefficients.

The second direction obtains congruences modulo 
 higher powers of $p$, by adding extra terms to the right hand side
 of Wolstenholme's congruence.
In 2000  van Hamme \cite{vH00} proved a result implying 
 that for all primes $p \geq 7$, 
\begin{equation}
\label{vanhamme}
{2p-1\choose p-1}\equiv 1+2p\sum_{j=1}^{p-1}\frac{1}{j} \pmod{p^5},
\end{equation}
where $H_n:=\sum_{j=1}^n\frac{1}{j}$ are the harmonic numbers.
Recently Me\v strovi\' c \cite{Mes11} showed that for any prime $p\geq 11$,
\begin{equation}
\label{mestrovic}
{2p-1\choose p-1}\equiv 1-2p\sum_{j=1}^{p-1}\frac{1}{j}+4p^{2}\sum_{1\leq i<j\leq p-1}\frac{1}{ij}\pmod{p^{7}}
\end{equation}
This congruence involves the  additional expression
\[
\sum_{1\leq i<j\leq p-1}\frac{1}{ij},
\]
which is an example of a multiple harmonic sum, defined below.

The main result  of this paper is a simultaneous generalization and unification of these results,
 giving congruences for ${kp-1\choose p-1}$ to arbitrary powers of $p$, which involve multiple harmonic sums. 
Our basic method formulates and exploits the existence of families of
linear relations between certain multiple harmonic sums.
The coefficients in our congruences are given
by certain polynomials $b_{j,n}(T)$, defined below.

%
%

\subsection{Main result}\label{sec11}

A  \emph{composition} is a finite ordered list $\{\lambda_{1},\ldots,\lambda_{j}\}$ of positive integers. For ease of notation, we will denote by $\{\lambda_{1},\ldots,\lambda_{j}\}^a$ the composition
\[
\ul=\{\underbrace{\lambda_1,\ldots,\lambda_j},\underbrace{\lambda_1,\ldots,\lambda_j},\ldots,\underbrace{\lambda_1,\ldots,\lambda_j}\}
\]
consisting of $a$ concatenated copies of $\{\lambda_1,\ldots,\lambda_j\}$.

\begin{definition}
{\em
For a composition $\ul=\{\lambda_1,\ldots,\lambda_j\}$, and a positive integer $n$, we define the \emph{multiple harmonic sum}
\[
\H(\ul;n) :=\sum_{n\geq i_{1}>\ldots>i_{j}\geq 1}\frac{1}{i_{1}^{\lambda_1}\cdot\ldots\cdot i_{j}^{\lambda_j}}.
\]
(By convention, if $\ul=\{\lambda_1,\ldots,\lambda_j\}$ and $n<j$, we set $\H(\ul;n)=0$)
}
\end{definition}

For $k$ an integer, it has long been known (see e.g.\ \cite{Leh38}) that the binomial coefficient ${kp-1\choose p-1}$ can be written as a linear combination of `elementary symmetric' multiple harmonic sums $\H(\{1\}^j;p-1)$:
\[
{kp-1\choose p-1}=\sum_{j=0}^{p-1}(k-1)^jp^j\H(\{1\}^j;p-1)
\]
For a fixed non-negative integer $n$, we may truncate this equation after the $2n$-th term and use estimates on the $p$-divisibility of the multiple harmonic sums $\H(\{1\}^j;p-1)$ to obtain the congruence
\begin{equation}
\label{easycong}
{kp-1\choose p-1}\equiv\sum_{j=0}^{2n}(k-1)^jp^j\H(\{1\}^j;p-1)\pmod{p^{2n+3}}
\end{equation}
which holds for all primes $p\geq 2n+5$.

In Section 2 we show that the generating function for the elementary symmetric multiple harmonic sums $\H(\{1\}^j;p-1)$ (with $p$ fixed) satisfies a functional equation, which we use to derive identities involving these sums.  These identities can be used to cancel some of the terms appearing in \eqref{easycong}. Our main result (like equation \eqref{easycong}) gives for each non-negative integer $n$ a congruence for the binomial coefficient ${kp-1\choose p-1}$ mod $p^{2n+3}$, involving multiple harmonic sums.  However, our congruence involves only the first $n$ elementary symmetric multiple harmonic sums (instead of the first $2n$ such sums).

The coefficients in our congruences are given by polynomials in $k$.  We make the following definition.

\begin{definition}
{\em
Let $0\leq j\leq n$ be integers.  We define the \emph{\optimal polynomial} $b_{j,n}(T)\in\Q[T]$ to be the 
unique polynomial of degree at most $2n+1$ satisfying the following conditions:
\begin{enumerate}
\item[(C1)] $b_{j,n}(T)\equiv(T-1)^j\mod (T-1)^{n+1}$
\smallskip
\item[(C2)]
 $b_{j,n}(T)\equiv(-1)^jT^j\mod T^{n+1}$
\end{enumerate}
}
\end{definition}

A table of the \optimal  polynomials $b_{j,n}(T)$ for $0\leq j\leq n\leq 3$ can  be found in Section \ref{sec13}.
Now we can state our main result:

\begin{theorem} \label{th-main}{\rm (Optimized Congruences)}
Let $n \geq 0$ be a fixed integer. The \optimal polynomials $b_{j,n}(T)$ ($j=0,1,\ldots,n$) have integer coefficients, and the following hold:
\begin{enumerate}
\item For every prime $p \ge 2n+5$ and every integer $k\ge 1$:
\begin{equation}
\label{MTeq}
{kp -1 \choose p-1} \equiv\sum_{j=1}^n b_{j,n}(k) \,p^j\H( \{ 1\}^j; p-1)\pmod{p^{2n+3}}.
\end{equation}

\item If $p=2n+3$ is prime, then for every integer $k \ge 1$,  the above congruence 
holds (\rm{mod} ${p^{2n+2}}$).

\item For every prime $3 \le p\le 2n+1$ and every integer $k \ge 1$, the above congruence is equality: 
\[
{kp -1 \choose p-1} = \sum_{j=1}^n b_{j,n}(k) \,p^j\H( \{ 1\}^j; p-1).
\]

\end{enumerate}
\end{theorem}

Wolstenholme's congruence is the case  $n=0, k=2$ of Theorem \ref{th-main}.
As another  example, taking $n=k=3$  gives the congruence
\begin{eqnarray*}
{3p-1\choose p-1}\equiv &1&+402\,p\H(\{1\};p-1)-396\,p^2\H(\{1\}^2;p-1)\\
&+&216\,p^3\H(\{1\}^3;p-1)\pmod{p^9}.
\end{eqnarray*}
\smallskip

Theorem \ref{th-main}  has four important features:
\begin{enumerate}
\item The coefficients $b_{j,n}(k)$ appearing in the congruence \eqref{MTeq} are independent of the prime $p$. 

\item There are a large number of congruences for ${kp-1\choose p-1}$ holding mod $p^{2n+3}$, which involve the multiple harmonic sums $\H(\{1\}^j;p-1)$ for $1\leq j\leq 2n$ (see Theorem \ref{gwc}).  The congruences \eqref{MTeq} are optimized 
among these in only containing  the terms $\H(\{1\}^j;p-1)$ for $1\leq j\leq n$.

\item The restriction of the theorem to exclude certain small primes, depending on $n$, is necessary. 
 The congruences may fail to hold $(\bmod ~p^{2n+3})$ for
 $p=2n+3$ (when $2n+3$ is prime),
 and also fail to hold for $p=2$.

\item The  \optimal polynomials $b_{j,n}(T)$ depend on $n$, and 
for fixed $j$ their values at integers
$b_{j, n}(k)$, which are the coefficients in the congruences, 
do {\em not} stabilize as $n\to\infty$ (with the exception of $b_{0,n}(k)$; see the tables in Section \ref{sec13}). However they do satisfy many interesting congruences as $n$ varies, which we address in Section \ref{sec5}.
\end{enumerate}

One may ask whether the  coefficients $b_{j,n}(k)$ appearing in the \optimal congruences \eqref{MTeq} are uniquely characterized by \eqref{MTeq} holding for all sufficiently large primes $p$; we  conjecture this is the case,
and discuss it  in  Section \ref{sec12}.

An {\em exceptional congruence}  will be a triple $(k,n, p)$  such that the corresponding
congruence 
given in Theorem \ref{th-main} holds modulo an extra power of $p$.
  We  characterize exceptional
 congruences  for primes $p \ge 2n+3$ as follows.

\begin{theorem} {\em (Exceptional Congruences)}
\label{th-extra}
Let $n$ be a non-negative integer, $p$ an odd prime. 
For all $k\in\Z$, we have the following:
\begin{enumerate}[(i)]
\item Suppose $p\geq 2n+5$.  The exceptional congruence
\[
{kp-1\choose p-1}\equiv \sum_{j=0}^nb_{j,n}(k)p^j\H(\{1\}^j;p-1)\pmod{p^{2n+4}}
\]
holds if and only if either $k \equiv 0, 1 ~(\bmod ~p)$
or $p$ divides the numerator of the Bernoulli number $B_{p-2n-3}$.

\item  Suppose $p=2n+3$.  The exceptional congruence
\[
{kp-1\choose p-1}\equiv \sum_{j=0}^nb_{j,n}(k)p^j\H(\{1\}^j;p-1)\pmod{p^{2n+3}}
\]
holds if and only if $k\equiv 0,1\pmod{p}$.
\end{enumerate}
\end{theorem}

We obtain Theorem \ref{th-main} and Theorem \ref{th-extra} as special instances of a family 
$\cF_{N,k}$ of generalized Wolstenholme congruences,  given in Theorem \ref{gwc}, and with
the $\cF_{N, k}$ specified in Definition \ref{de34}.
The general congruence in the family $\cF_{N,k}$ 
(which will hold for all sufficiently large primes $p$) is of the form
\begin{equation}
\label{genfam}
{kp-1\choose p-1}\equiv\sum_{j=0}^{N}b_jp^j\H(\{1\}^j;p-1)\pmod{p^{N+1+\epsilon}}
\end{equation}
where $\epsilon\in\{1,2\}$ is chosen so that $\epsilon\equiv N\pmod{2}$, and the coefficients $b_j$ are rational numbers.
Each congruence in this general family is derived from \eqref{easycong}, using 
linear combinations of identities among multiple harmonic sums (these identities are  stated as Theorem \ref{thrm}).
The optimized congruence \eqref{MTeq} is distinguished as the unique congruence
 in the family $\cF_{2n,k}$ satisfying $b_{n+1}=b_{n+2}=\ldots=b_{2n}=0$.

%
%
\subsection{The \optimal polynomials $b_{j,n}(T)$}\label{sec13}

In Section \ref{sec5} we prove some interesting properties of the \optimal polynomials $b_{j,n}(T)$.
Here we  present  data on these polynomials for small $j, n$ in Table 1 below.


\begin{table}[h]
\label{tab1}
\caption{\Optimal Polynomials  $b_{j,n}(T)$ }
\centering
\begin{tabular}{|c||c|c|c|c|c|c|}
\hline $n\backslash j$&0&1&2&3\\ \hline\hline
0&~1~&&&\\ \hline
1&~1~&$T^2-T$&&\\ \hline
2&~1~&$-T^4+2T^3-T$&$T^4-2T^3+T^2$&\\ \hline
3&~1~&$2T^6-6T^5+5T^4$&$-2T^6+6T^5-5T^4$& $T^6-3T^5+3T^4$\\
~&~&$~~~~~~~ {}{}-T+1$&  $~~~~~~~~~~~~~~ +T^2$&$~~~~~~~ -T^3$ \\ \hline
\end{tabular}

\end{table}

 Table 1 illustrates that $b_{0,n}(T)=1$ for all $n$ (this will be established in Section \ref{sec5}).  
While  the definition of $b_{j,n}(T)$ given earlier shows that it has  degree at most $2n+1$, in fact its
degree is at most $2n$ (see Theorem \ref{recipe}).

We next consider the  coefficients $b_{j,n}(k)$ appearing in the  \optimal congruences given in Theorem \ref{th-main}.  Values of the coefficients for $k=2$ and $k=3$ are given in  Table 2 (resp.\ Table 3)  below.\bigskip

\begin{table}[h]
\caption{Values $b_{j,n}(k)$ for $k=2$}
\centering
\begin{tabular}{|c||c|c|c|c|c|c|}
\hline $n\backslash j$&0&1&2&3&4&5\\ \hline\hline
0&1&&&&&\\ \hline
1&1&2&&&&\\ \hline
2&1&-2&4&&&\\ \hline
3&1&14&-12&8&&\\ \hline
4&1&-66&68&-40&16&\\ \hline
5&1&382&-380&248&-112&32\\ \hline
\end{tabular}

\end{table}

\bigskip
Table 2 shows that  $b_{1,2}(2)=-2$, $b_{2,2}(2)=4$, so that Theorem \ref{th-main} reduces to
Me\v strovi\' c's  result \eqref{mestrovic} in the case $n=k=2$.

\begin{table}[h]
\label{tab3}
\caption{Values $b_{j,n}(k)$ for $k=3$}
\centering
\begin{tabular}{|c||c|c|c|c|c|c|}
\hline $n\backslash j$&0&1&2&3&4&5\\ \hline\hline
0&~1~&&&&&\\ \hline
1&~1~&6&&&&\\ \hline
2&~1~&-30&36&&&\\ \hline
3&~1~&402&-396&216&&\\ \hline
4&~1~&-6078&6084&-3672&1296&\\ \hline
5&~1~&102786&-102780&66312&-29808&7776\\ \hline
\end{tabular}
\end{table}

Tables 2 and 3  illustrate that for $j \ge 1$, the  
coefficients $b_{j,n}(k)$ which appear in the congruences \eqref{MTeq}
do not appear to stabilize as $n \to \infty$ (holding $j$ and $k$ fixed).

%
%

\subsection{Uniqueness conjecture}\label{sec12}

 The statement of Theorem \ref{th-main} raises an issue concerning  whether 
 the coefficients $b_{j,n}(k)$ above are uniquely determined 
 by the condition that the congruences  \eqref{MTeq}
 hold for all sufficiently large primes $p$.  Theorem \ref{MT} asserts that there is a unique congruence of the form \eqref{MTeq} in the general family $\cF_{N,k}$. 
We believe that our family $\cF_{N,k}$ actually contains {\em all} congruences of the shape 
\eqref{genfam} which hold for all sufficiently large primes $p$, in which case 
the individual coefficients  $b_{j,n}(k)$ above would be uniquely determined by \eqref{MTeq}, 
but we do not establish this.
We formulate this belief as  the following conjecture.


\begin{conjecture}\label{cj-main} {\rm (Uniqueness Conjecture)}

Let $n\geq0$, $k$ be integers, $b_0,b_1,\ldots,b_n,\in\Q$ such that the congruence
\[
{kp-1\choose p-1}\equiv \sum_{j=0}^nb_jp^j\H(\{1\}^j;p-1)\pmod{p^{2n+2}}
\]
holds for all sufficiently large primes $p$. Then $b_j=b_{j,n}(k)$ for $j=0,1,\ldots,n$.
\end{conjecture}

This conjecture might be difficult to resolve, in view of the following consequence.

\begin{proposition}
\label{prop-imp}
In the special case $n=k=1$, the Uniqueness Conjecture \ref{cj-main} implies that there are infinitely many primes $p$ such that  the 
numerator of the Bernoulli number $B_{p-3}$ is not divisible by $p$.
\end{proposition}

This  property of the Bernoulli numbers is currently an open problem. 

\begin{proof}
Take $n=k=1$, $b_0=b_1=1$.  These values do not agree with the values of the \optimal coefficients $b_{0,1}(1)=1$, $b_{1,1}(1)=0$, so the Uniqueness Conjecture 
states that there are infinitely many primes $p \ge 7$ for which
\[
1= {p-1\choose p-1}\not\equiv 1+p\H(\{1\};p-1)\pmod{p^4}.
\]
That is, there are infinitely many primes $p$ for which $p^3\nmid\H(\{1\};p-1)$. 
Glaisher \cite{Gla1900a} showed that, for all primes $p\geq5$, 
\[
\H(\{1\};p-1)\equiv -\frac{B_{p-3}}{3}p^2\pmod{p^3}.
\]
The result now follows.
\end{proof}

In Section 3 we formulate a more general Strong Uniqueness Conjecture \ref{uniq},
which we show in Section 4 implies the Uniqueness Conjecture.

\subsection{Extensions of  results}\label{sec14} 

Theorem \ref{th-extra} connects certain exceptional  congruences with $p$-divisibility
of the numerators of certain Bernoulli numbers. 
Recall that theorems of Herbrand and Ribet say that for odd $i$ in the range $3\leq i\leq p-2$, $p|B_{p-i}$ if and only if a particular piece of the class group of the cyclotomic field $\Q(\zeta_p)$ has non-trivial $p$-part (see \cite{Was82}, Sec.\ 6.3 for a precise statement). This raises the possiibility that our congruences
may have an interpretation in terms of the arithmetic of cyclotomic fields. 
We do not currently know if that is the case, but 
in \cite{Ro12b} we will show that there is a parallel family of congruences related to $p$-adic $L$-function
values at positive integers, which involves the `power' multiple harmonic sums $H(\{ j\}; p-1)$.
In \cite{Ro12} we systematically investigate the structure of identities among
multiple harmonic numbers that underly such congruences. 

\subsection{Related results}\label{sec15} 
The literature contains a vast
collection of identities and congruences involving multiple harmonic sums and related sums,
starting with work of Euler on harmonic numbers. Some of these involve
binomial coefficient congruences (see  Granville \cite{Gran95} for a survey).
Our generalized congruences  appear to have a structure not observed before, but for convenience
we summarize some results from the literature for comparison.

A number of congruences  are known
for the elementary symmetric multiple harmonic sums $H( \{ 1\}^{r}; n)$ considered
in this paper.
In 1900 Glaisher \cite{Gla1900a} proved that 
 for all odd $r \geq5$ and all primes $p\geq7$,
\[
S_{r}(p) := \frac{pr}{2}\H(\{1\}^r;p-1)-\H(\{1\}^{r-1};p-1)\equiv 0\pmod{p^4},
\]
holds.  
In 1953 Carlitz \cite{Car53b} sharpened the congruence  of Glaisher 
to show for all odd $r \ge 5$ and prime $p \ge 7$,
\begin{eqnarray*}
S_{r}(p) \equiv
p^4\frac{(p-r)(p-r-1)(p-r-2)}{24(p-r-3)(p-1)!}B_{p-3}\pmod{p^5}, 
\end{eqnarray*}
giving a relation with Bernoulli numbers.  Along similar lines, Tauraso \cite{Tau10} shows
  that for any prime $p\geq7$,
\begin{eqnarray*}
\H(\{1\};p-1)\hspace{-.1cm}&\equiv&\hspace{-.1cm} -\frac{1}{2}p\H(\{1\}^2;p-1)-\frac{1}{6}p^2\H(\{1\}^3;p-1)\\
&\equiv&\hspace{-.1cm}p^2\left(\frac{B_{2p-5}}{3p-5}-3\frac{B_{2p-4}}{2p-4}+3\frac{B_{p-3}}{p-3}\right)+p^4\frac{B_{p-5}}{p-5}\hspace{-.2cm}\pmod{p^5}
\end{eqnarray*}

There are also many congruences known that  involve power sum multiple harmonic sums 
\[
H( \{ r\}, p-1) = \sum_{j=1}^{p-1} \frac{1}{j^r}.
\]
Washington \cite{Was98} provided a formula expressing 
these sums in terms of values of $p$-adic L-functions at positive integers.  
Congruences modulo $p$ involving arbitrary multiple harmonic sums have been investigated by Hoffman \cite{Hof04a}.  

Multiple harmonic sums appear in certain calculations in physics. 
 Bl\"{u}mlein (\cite{Blu03} discussed applications of multiple harmonic sums to quantum field theory.  
 He computed a family of algebraic relations between multiple harmonic sums $\H(\ul;n)$ 
 which are independent of the upper limit of summation $n$.
 The Hopf algebra of quasi-symmetric functions has also been used by Hoffman to investigate both multiuple harmonic sums \cite{Hof04} and multiple zeta values \cite{Hof97}.

 Kontsevich \cite{Kon02} considered the function
\[
H_n(x) = \sum_{j=1}^n \frac{x^j}{j},
\]
which is a `twisted' version of a multiple harmonic sum.  The expression $H_n(x)$ may also be considered as a trancated version of the series $-\log(1-x)$.
The corresponding sums for the dilogarithm $\sum_{j=1}^n \frac{x^j}{j^2}$
 were considered by Elbaz-Vincent and Gangl \cite{EVG02}, who
 determine functional equations satisfied by such series.
 These values are truncated forms of multiple zeta values, which  is of great current interest (see Zagier \cite{Zag94},
 Brown \cite{Bro12}).

\section{Representing Binomial Coefficients in Terms of Multiple Harmonic Sums}

Our first object is to express binomial coefficients in terms of multiple harmonic sums.
For $n$ a positive integer, define the  polynomial
\[
f_{n}(T) ={(n+1)(T+1)-1\choose n}=\frac{1}{n!}\prod_{j=1}^n ((n+1)T+j)
\]
It will be useful to rewrite $f_n$ in the form
\[
f_n(T)=\prod_{j=1}^n\left(1+\frac{n+1}{j}T\right)
\]
Expanding the product above, we see that the coefficient of $T^j$ is $(n+1)^j$ times the $j$-th elementary symmetric function in $1,\frac{1}{2},\frac{1}{3},\ldots,\frac{1}{n}$.  In other words, we have
\begin{equation}
\label{harmf}
f_{n}(T) =\sum_{j=0}^n(n+1)^j\H(\{1\}^j;n)T^j
\end{equation}

By convention, we take $H(\{1\}^{j};n)=0$ when $j<0$ or $j>n$.  We also take $\H(\phi;n)=1$.

The polynomial $f_n$ satisfies a functional equation relating $T$ and $-1-T$.  We can compute

\begin{eqnarray*}
f_n(-1-T)&=&\frac{1}{n!}\prod_{j=1}^n ((n+1)(-1-T)+j)\\
&=&\frac{(-1)^n}{n!}\prod_{j=1}^n ((n+1)T+n+1-j)\\
&=&\frac{(-1)^n}{n!}\prod_{i=1}^n ((n+1)T+i)\\
&=&(-1)^nf_n(T)
\end{eqnarray*}
(In the third line we have made the substitution $i=n+1-j$).
We expand this functional equation using \eqref{harmf} to get

\begin{eqnarray*}
\sum_{j\geq 0}(n+1)^{j}H(\{1\}^{j};n)T^{j}&=&(-1)^n\sum_{j\geq 0}(n+1)^{j} H(\{1\}^{j};n)(-1-T)^{j}\\
&=&\sum_{j\geq 0}(-1)^{n+j}(n+1)^{j} H(\{1\}^{j};n)\sum_{0\leq i\leq j}{j\choose i}T^{i}\\
&=&\sum_{i\geq 0}\left(\sum_{j\geq i}{j\choose i}(-1)^{n+j} (n+1)^{j} H(\{1\}^{j};n)\right)T^{i}
\end{eqnarray*}
This holds identically in $T$.  Equating the coefficient of $T^{j}$ on each side of the preceding equality and rearranging gives
the following identity. 

\begin{proposition}
For all non-negative integers $n$, $j$, we have
\begin{equation}\label{rep0}
(n+1)^{j} \H(\{1\}^{j};n)+\sum_{i\geq j}(-1)^{n+i+1}{i\choose j}(n+1)^{i} \H(\{1\}^{i};n)=0
\end{equation}
\end{proposition}
The sum above is finite (terms corresponding to $i>n$ vanish). We thus have a family of linear equations (indexed by $j$) satisfied by the quantities $(n+1)^i\H(\{1\}^i;n)$, $i=0,1,\ldots,n$.

From the above we obtain a general set of identities  expressing binomial coefficients
in terms of  the $ H(\{1\}^{j};n)$.

%
\begin{proposition}\label{thrm}
Let $n$ be a non-negative integer, $k$, $c_{0},c_{1},\ldots$ indeterminates, and define
\begin{equation}
\label{defb}
b_{j}:=(k-1)^{j}+c_{j} + (-1)^{n+j+1}\sum_{i=0}^{j}{j\choose i}c_{i}.
\end{equation}
Then the equation
\begin{equation}
\label{eq}
{k(n+1)-1\choose n}=\sum_{j=0}^{\infty} b_{j}(n+1)^{j} H(\{1\}^{j};n),
\end{equation}
holds identically in the indeterminates $k$, $c_0,c_1,\ldots$. Here the  right  side of \eqref{eq} 
is a finite sum, since $H(\{1\}^j; n)=0$ for $j >n$.
\end{proposition}

\begin{proof}
To begin, we use Equation \eqref{harmf} to write
\begin{eqnarray*}
{k(n+1)-1\choose n}&=&f_n(k-1)\\
&=&\sum_{j\geq0}(k-1)^{j}(n+1)^{j} H(\{1\}^{j};n)
\end{eqnarray*}
Considering $n$ fixed, we add to this equation a linear combination of equations \eqref{rep0}
 (where $c_{j}$ is the coefficient of the equation indexed by $j$) to obtain the general formula.
\end{proof}

\begin{remark}
{\em By making suitable choices of the parameters $c_i$ in 
Proposition \ref{thrm}, we can arrange to have $b_j=0$ for many $j$. Theorem \ref{MT} is obtained by optimizing this process. 
Special cases of the identities in this proposition were noted long ago.
For example, Emma Lehmer (\cite{Leh38}, p.\ 360) used the identity
\[
{p-1 \choose k} = (-1)^k \sum_{j=0}^k (-1)^j p^j \H( \{ 1\}^j; k).
\]
This  particular identity is valid for all integers $m$, namely 
\[
{m \choose k} = (-1)^k  \sum_{j=0}^k (-1)^j m^j \H( \{ 1\}^j;k).
\]
However most identities produced above hold only for restricted
values of $m$, namely $m=k(n+1)-1$, for fixed $n$. }
\end{remark}

\section{Congruences for  ${kp-1\choose p-1}$ Modulo Powers of $p$}

To obtain congruences for ${kp-1\choose p-1}$, we will take $k$ to be an integer and truncate the expansion of 
Proposition \ref{thrm}, with $n=p-1$, $p$ a prime.  To establish a bound on the error due to truncation, we need to establish some congruence properties of multiple harmonic sums.

In the remainder of the paper, when $p$ is understood to be a fixed prime, we sometimes abbreviate 
\[
\H(\{1\}^j):= \H(\{1\}^j;p-1).
\]

\subsection{Congruence properties of multiple harmonic sums}

Zhao (\cite{Zh08}, Theorem 1.6) gives the following congruence involving multiple harmonic sums $\H(\{1\}^j;p-1)$ and Bernoulli numbers:

\begin{proposition}
\label{scc}
Let $p$ be a fixed odd prime, and $j$ an integer with $1\leq j\leq p-3$.  Then we have
\[
\H(\{1\}^j; p-1)\equiv\begin{cases}\frac{-B_{p-1-j}}{j+1}p\pmod{p^2}~~~~~\text{ if }j\equiv 0\mod{2}\\
\left(\frac{-(j+1)}{2(j+2)}B_{p-2-j}\right)p^2\pmod{p^3}\text{ if }j\equiv 1\mod{2}\\
\end{cases}
\]
\end{proposition}

We provide additional congruence for $\H(\{1\}^j); p-1)$ for those $j$ which are not covered by Proposition \ref{scc}:

\begin{proposition}
\label{propextra}
Let $p$ be a fixed odd prime, and $j$ a positive integer.
\begin{enumerate}[(i)]
\item If $j=p-2$ we have $\H(\{1\}^j; p-1)\equiv\frac{1}{2}p\pmod{p^2}$.
\item If $j=p-1$, we have $\H(\{1\}^j); p-1) \equiv-1\pmod{p}$.
\item If $j\geq p$, we have $\H(\{1\}^j; p-1)=0$.
\end{enumerate}
\end{proposition}

\begin{proof}
\begin{enumerate}[(i)]
\item We have
\begin{eqnarray*}
\H(\{1\}^{p-2})&=&\sum_{i=1}^{p-1}\frac{1}{1\cdots \hat{i}\cdots(p-1)}\\
&=&\frac{1}{(p-1)!}\sum_{i=1}^{p-1}i\\
&=&\frac{1}{(p-1)!}\frac{p(p-1)}{2}\\
&\equiv&\frac{p}{2}\pmod{p^2}
\end{eqnarray*}
where in the last line, we have used Wilson's theorem, which states that $(p-1)!\equiv-1\pmod{p}$.
\item We have
\begin{eqnarray*}
\H(\{1\}^{p-1})&=&\frac{1}{(p-1)!}\\
&\equiv&-1\pmod{p}
\end{eqnarray*}
\item In this case the defining sum is empty.
\end{enumerate}
\end{proof}

\subsection{A general family of congruences}

We can now obtain our general family of congruences for ${kp-1\choose p-1}$.  The congruences are obtained from Proposition \ref{thrm} by truncation. We take some care to epress the error due to truncation in terms of Bernoulli numbers.

\begin{theorem}[General Wolstenholme-like Congruence]
\label{gwc}
Let $k$ be an integer.  Let $c_{0},c_{1},\ldots\in\Q$ be given, and take $b_{j} \in \Q$ defined by 
\begin{equation}
\label{defb}
b_j=(k-1)^j+c_j+(-1)^{j+1}\sum_{i=0}^j {j\choose i}c_i
\end{equation}
Fix an odd prime $p$ which does not divide the denominator of any $c_i$, and let $N$ be a non-negative integer.  Define
\[
E_N:={kp-1\choose p-1}-\sum_{j=0}^{N}b_{j}p^{j} \H(\{1\}^{j}; p-1)
\]
\begin{enumerate}
\item[\textbf{(i)}] If $0\leq N\leq p-4$, we have
\[
E_N\equiv\frac{-B_{p-3-N}}{N+3}\left(\frac{N+2}{2}b_{N+1}+b_{N+2}\right)p^{N+3}\pmod{p^{N+4}}
\]
when $N$ is even, and
\[
E_N\equiv\frac{-B_{p-2-N}}{N+2}b_{N+1}p^{N+2}\pmod{p^{N+3}}
\]
when $N$ is odd.

\item[\textbf{(ii)}] If $N=p-3$, we have
\[
E_N\equiv\left(\frac{b_{N+1}}{2}-b_{N+2}\right)p^{N+2}\pmod{p^{N+3}}
\]
\item[\textbf{(iii)}] If $N=p-2$, we have
\[
E_N\equiv-b_{N+1}p^{N+1}\pmod{p^{N+2}}
\]
\item[\textbf{(iv)}] If $N\geq p-1$, we have
\[
E_N=0
\]
\end{enumerate}
In particular this implies that
\begin{equation}
\label{gwceq}
{kp-1\choose p-1}\equiv\sum_{j=0}^N b_jp^j\H(\{1\}^j) \begin{cases} \mod{p^{N+3}}\text{ if }N\leq p-4,\text{ $N$ even}\\
\mod{p^{N+2}}\text{ if }N\leq p-4,\text{ $N$ odd}\\
\mod{p^{N+2}}\text{ if }N=p-3\\
\mod{p^{N+1}}\text{ if }N=p-2\\
\mod{0}\text{ if }N\geq p-1
\end{cases}
\end{equation}
(Congruence mod 0 means equality)

\end{theorem}

\begin{proof} We apply Theorem \ref{thrm} to obtain the equality
\[
{kp-1\choose p-1}=\sum_{j=0}^{\infty}b_{j}p^{j}\H(\{1\}^{j}).
\]
Because $n=p-1$ is even, the values of $b_i$ are independent of $p$.  It follows that
\[
E_N=\sum_{j=N+1}^{\infty}b_{j}p^{j}\H(\{1\}^{j})
\]
We use Proposition \ref{propextra}, statement \emph{(iii)} to eliminate the terms in the above some with $j\geq p$, giving
\begin{equation}
\label{errorterm}
E_N=\sum_{N+1\leq j\leq p-1}b_{j}p^{j}\H(\{1\}^{j})
\end{equation}
The $b_{j}$ are $p$-integral (because we assume the $c_{j}$ to be $p$-integral). Propositions \ref{scc} and \ref{propextra} also tell us that the expressions $\H(\{1\}^j)$ are $p$-integral. We now separate into cases.
\smallskip

\textbf{Case (i-a)}: Suppose $0\leq N\leq p-5$ and $N$ is even.  Due to the integrality of $b_j$ and $\H(\{1\}^j)$, we have
\begin{eqnarray*}
E_N\equiv b_{N+1}p^{N+1}\H(\{1\}^{N+1})&+&b_{N+2}p^{N+2}\H(\{1\}^{N+2})\\
&+&b_{N+3}p^{N+3}\H(\{1\}^{N+3})\pmod{p^{N+4}}
\end{eqnarray*}
Proposition \ref{scc} tells us that $\H(\{1\}^{N+1})\equiv\frac{-(N+2)}{2(N+3)}B_{p-3-N}p^2\pmod{p^3}$ and $\H(\{1\}^{N+2})\equiv\frac{-1}{N+3}p\pmod{p}$.  Also, Proposition \ref{scc} or Proposition \ref{propextra} give that $\H(\{1\}^{N+3})\equiv0\pmod{p}$.  Combining these give
\[
E_N\equiv\frac{-B_{p-3-N}}{N+3}\left(\frac{N+2}{2}b_{N+1}+b_{N+2}\right)p^{N+3}\pmod{p^{N+4}},
\]
as desired.
\smallskip

\textbf{Case (i-b)}: Suppose $0\leq N\leq p-5$ and $N$ is odd.  Due to the integrality of $b_j$ and $\H(\{1\}^j)$, we have
\[
E_N\equiv b_{N+1}p^{N+1}\H(\{1\}^{N+1})+b_{N+2}p^{N+2}\H(\{1\}^{N+2})\pmod{p^{N+3}}
\]
Proposition \ref{scc} tells us that $\H(\{1\}^{N+1})\equiv\frac{-1}{N+2}B_{p-2-N}p\pmod{p^2}$ and $\H(\{1\}^{N+2})\equiv 0\pmod{p^2}$.   This gives
\[
E_N\equiv \frac{-B_{p-2-N}}{N+2}b_{N+1}p^{N+2}\pmod{p^{N+3}},
\]
as desired.
\smallskip

\textbf{Case (i-c)}: Suppose $N=p-4$.  Due to the integrality of $b_j$ and $\H(\{1\}^j)$, we have
\[
E_N\equiv b_{N+1}p^{N+1}\H(\{1\}^{N+1})+b_{N+2}p^{N+2}\H(\{1\}^{N+2})\pmod{p^{N+3}}
\]
Proposition \ref{scc} tells us that $\H(\{1\}^{N+1})\equiv\frac{-1}{N+2}B_{p-2-N}p\pmod{p^2}$.  Proposition \ref{propextra} tells us that $\H(\{1\}^{N+2})\equiv \frac{1}{2}p\pmod{p^2}$. Combining these gives
\[
E_N\equiv \frac{-B_{p-2-N}}{N+2}b_{N+1}p^{N+2}\pmod{p^{N+3}}
\]
\smallskip

\textbf{Case (ii)}: Suppose $N=p-3$.  In this case \eqref{errorterm} becomes
\[
E_N=b_{N+1}p^{N+1}\H(\{1\}^{p-2})+b_{N+2}p^{N+2}\H(\{1\}^{p-1})
\]

Proposition \ref{propextra} gives $\H(\{1\}^{p-2})\equiv\frac{p}{2}\pmod{p^2}$ and $\H(\{1\}^{p-1})\equiv-1\pmod{p}$, so
\[
E_N\equiv\left(\frac{b_{N+1}}{2}-b_{N+2}\right)p^{N+2}\pmod{p^{N+3}}
\]
\smallskip

\textbf{Case (iii)}: Suppose $N=p-2$.  In this case \eqref{errorterm} becomes
\[
E_N=b_{N+1}p^{N+1}\H(\{1\}^{p-1})
\]
Proposition \ref{propextra} gives $\H(\{1\}^{p-1})\equiv-1\pmod{p}$, so
\[
E_N\equiv -b_{N+1}p^{N+1}\pmod{p^{N+2}}
\]
\smallskip

\textbf{Case (iv)}: Suppose $N\geq p-1$.  Then \eqref{errorterm} shows that $E_N=0$.

\end{proof}

\begin{definition}\label{de34}
{\em
We call the congruence \eqref{gwceq} the \emph{generalized Wolstenholme congruence} associated with the data
\[
\left[k,(c_{0},c_{1},\ldots),N\right],
\]
and we will say that $b_0,\ldots,b_N$ are the \emph{generalized Wolstenholme coefficients} associated with this data.

We let $\cF_{N,k}$ denote the family of all generalized Wolstenholme congruences above, where $N, k$ are fixed
and the other data varies.}
\end{definition}

\begin{remark}
{\em For fixed $k$, $N$, the family $\cF_{N,k}$ has the structure of an affine linear space over $\Q$ in the following way: if $B=(b_0,\ldots,b_N)$ and $B'=(b'_0,\ldots,b'_N)$ are the coefficients associated with the data $[k,(c_0,c_1,\ldots),N]$, $[k,(c'_0,c'_1,\ldots),N]$ respectively, and $t\in\Q$, then
\[
tB+(1-t)B'=(tb_0+(1-t)b'_0,\,\ldots,\,tb_N+(1-t)b'_N),
\]
where the numbers on right hand side are the generalized Wolstenholme coefficients associated with the data
\[
[k,(tc_0+(1-t)c'_0,\,tc_1+(1-t)c'_1,\,\ldots),N]
\]
In the next section we will focus exclusively on the case where $N=2n$ is even.  
We will determine that the affine space of generalized 
Wolstenholme coefficients, for arbitrary $k$ and $N=2n$, has dimension $n$.  The `optimized' congruence will be uniquely determined among this $n$ dimensional family.}
\end{remark}

Let $N$ be a positive integer and set $M=N+3$ if $N$ is even, $M=N+2$ if $N$ is odd. It might be reasonable to expect that every congruence of the form
\[
{kp-1\choose p-1}\equiv \sum_{j=0}^Nb_jp^j\H(\{1\}^j; p-1)\pmod{p^M}
\]
that holds for all sufficiently large primes, actually comes from Theorem \ref{gwc}.  We formulate 
this as the following conjecture:


\begin{conjecture} {\rm (Strong Uniqueness Conjecture)}
\label{uniq}
If $k,m$ are integers with $m\geq0$, and $a_{0},\ldots,a_{n}\in\Q$ are such that
\[
{kp-1\choose p-1}\equiv a_{0}+a_{1}p H(\{1\}^{1}; p-1)+\ldots+a_{n}p^{n} H(\{1\}^{n}; p-1)\pmod{p^{m}}
\]
holds for all but finitely many $p$, then this congruence arises from Theorem \ref{gwc}, in the following sense: there are constants $c_{0},c_{1},\ldots\in \Q$ such that, if $b_{0},b_{1},\ldots$ are defined  by \eqref{defb}, then we have $a_{i}=b_{i}$ for $i=0,1,\ldots,\psi(m)$, where $\psi(m)=m-2$ if $m$ is even and $\psi(m)=m-3$ if $m$ is odd (here, we take $a_{i}=0$ for $i>n$).
\end{conjecture}
Proposition \ref{uniqimp} below will show that this conjecture implies the Uniqueness Conjecture \ref{cj-main}.

Now we consider some special cases of Theorem \ref{gwc}.
In what follows we will write $(c_{0},\ldots,c_{m})$ for the sequence $(c_{0},\ldots,c_{m},0,0,\ldots)$.

As one example, fix a positive integer $k$ and take the data $\left[k,((k-1)^2),2\right]$.  This gives $(b_{0},b_{1},b_{2},b_{3},b_{4})=(1,k(k-1),0)$, so we get the congruence
\begin{corollary}
\label{sc1}
For all integers $k$ and all primes $p\neq 2,5$, we have
\[
{kp-1\choose p-1}\equiv 1+k(k-1)p\H(\{1\}; p-1)\pmod{p^{5}}
\]
\end{corollary}
This is a generalization of van Hamme's result \eqref{vanhamme}.

Taking the data $\left[2,(49,-18,4),6\right]$ gives\\
$(b_{0},b_{1},\ldots,b_{6})=(1,14,-12,8,0,0,0)$, so we get the identity

\begin{corollary}
\label{sc2}
For all odd primes $p$, we have
\begin{eqnarray*}
{2p-1\choose p-1} &\equiv & 1+14p H(\{1\}^{1}; p-1)-12p^{2} H(\{1\}^{2}; p-1)\\
&& ~~~~~+8p^{3} H(\{1\}^{3}; p-1)\pmod{p^{9}}
\end{eqnarray*}
\end{corollary}

Corollaries \ref{sc1} and \ref{sc2} are special cases of Theorem \ref{gwc}, corresponding to $n=1$ and $n=3$, respectively.

%
%

\section{Optimized Wolstenholme Congruences}
We now state  a version of our main result (Theorem \ref{th-main}).
We show that when $N=2n$ is even, it is always possible to choose the data 
$(c_0,c_1,\ldots)$ so that $b_{n+1}=b_{n+2}=\ldots=b_{2n}=0$.  Moreover, this condition will uniquely determine the values of $b_j$ for $0\leq j\leq n$. We will derive this result using Theorem \ref{gwc}.


\begin{theorem}[Optimized Wolstenholme Congruences]
\label{MT}
Let integers $k$, $n$ be given, with $n\geq0$, and set $N=2n$.  Then there exist unique values $b_{j,n}(k)\in\Q$ ($j=0,1,\ldots,n$) with the following property:

There exist $c_{0},c_{1},\ldots\in\Q$ such that the generalized Wolstenholme coefficients $b_0,b_1,\ldots,b_N$ associated with the data $[k,(c_{0},c_{1},\ldots),N]$ satisfy $b_{j}=b_{j,n}(k)$ for $0\leq j\leq n$, and
\[
b_{n+1}=b_{n+2}=\ldots=b_{2n}=0
\]
Additionally, the $c_j$ may be taken to be integers, so that the $b_{j,n}(k)$ are necessarily integers.

In other words, for $N=2n$, Theorem \ref{gwc} produces a unique congruence of the form
\begin{equation}
\label{goodp}
{kp-1\choose p-1}\equiv b_{0}+b_{1}p H(\{1\}^{1}; p-1)+\ldots+b_{n}p^{n} H(\{1\}^{n}; p-1)\pmod{p^{2n+3}}
\end{equation}
with $b_{i}\in\Z$, which holds for all odd primes $p\neq 2n+3$.

In the case $p=2n+3$ we have
\begin{equation}
\label{badp}
{kp-1\choose p-1}\equiv b_{0}+b_{1}p H(\{1\}^{1}; p-1)+\ldots+b_{n}p^{n} H(\{1\}^{n}; p-1)\pmod{p^{2n+2}}
\end{equation}
Additionally, for odd $p\leq2n+1$, we have
\[
{kp-1\choose p-1}= b_{0}+b_{1}p H(\{1\}^{1}; p-1)+\ldots+b_{n}p^{n} H(\{1\}^{n}; p-1)
\]
\end{theorem}
We will defer the proof of this result, first showing that it gives an implication between the two uniqueness conjectures we have made.

\begin{proposition}\label{uniqimp}
The  Strong Uniquenss Conjecture \ref{uniq} implies  
the Uniqueness Conjecture \ref{cj-main}.
\end{proposition}
\begin{proof}
 Let $n\geq0$, $k$ be given, and suppose $b_0,\ldots,b_n\in\Q$ are taken such that the congruence
\[
{kp-1\choose p-1}\equiv1+\sum_{j=1}^nb_jp^j\H(\{1\}^j;p-1)\pmod{p^{2n+2}}
\]
holds for all sufficiently large $p$.  Suppose the Strong Uniqueness
Conjecture \ref{uniq} is true, and apply it with $m=2n+2$, to find that there are $c_0,c_1,\ldots\in \Q$ such that
\[
b_j=(k-1)^j+c_j+(-1)^{j+1}\sum_{i=0}^j{j\choose i}c_i
\]
for $j=0,1,\ldots,n$, and
\[
(k-1)^j+c_j+(-1)^{j+1}\sum_{i=0}^j{j\choose i}c_i=0
\]
for $j=n+1,n+2,\ldots,2n$.  Now, the uniqueness statement of Theorem \ref{MT} above says that $b_j=b_{j,n}(k)$ for $j=0,1,\ldots,n$.
\end{proof}

For the proof of Theorem \ref{MT},  we need some preliminary definitions and lemmas.
\begin{definition}
\label{defV}
{\em
Fix integers $N$, $k$, with $N\geq0$. Define $V_{N,k}\subset \Z^{N+1}$ to be the set
\[
V_{N,k} :=\left\{(b_0,\ldots,b_N):\exists c_0,c_1,\ldots\in\Z\text{ s.t.\ }b_j=(k-1)^j+c_j+(-1)^{j+1}\sum_{i=0}^j {j\choose i}c_i\right\}
\]
In other words $V_{N,k}$ is the set of {\em generalized Wolstenholme coefficients} corresponding to integer data. We similarly define $V_{N,k}^\Q\subset\Q^{N+1}$ to be
\[
V_{n,K}^\Q :=\left\{(b_0,\ldots,b_N):\exists c_0,c_1,\ldots\in\Q\text{ s.t.\ }b_j=(k-1)^j+c_j+(-1)^{j+1}\sum_{i=0}^j {j\choose i}c_i\right\},
\]
the set of generalized Wolstenholme coefficients corresponding to rational data.}
\end{definition}
The inclusion $V_{N,k}\hookrightarrow V_{N,k}^\Q$ induces an isomorphism
\[
V_{N,k}\otimes\Q\cong V_{N,k}^\Q
\]
of affine spaces over $\Q$.

We have that $V_{N,k}$ is a coset of the subgroup
\[
\hat{V}_{N}:=\left\{(b_0,\ldots,b_N):\exists c_0,c_1,\ldots\in\Z\text{ s.t.\ }b_j=c_j+(-1)^{j+1}\sum_{i=0}^j {j\choose i}c_i\right\}\subseteq\Z^N
\]
Note that $\hat{V}_N$ is independent of $k$.  We then have the following:

\begin{proposition}
\label{swap}
For all integers $N$, $k$, with $N\geq0$, we have $V_{N,k}=V_{N,1-k}$.
\end{proposition}

This is not surprising, as
\[
{kp-1\choose p-1}={(1-k)p-1\choose p-1}
\]
holds for all odd primes $p$ and all $k$.
\begin{proof}
As $V_{N,k}$ and $V_{N,1-k}$ are cosets of the same subgroup $\hat{V}_N\leq \Z^{N+1}$, equality will follow if we can show $V_{N,k}\cap V_{N,1-k}\neq\phi$.

Taking $c_0=c_1=\ldots=0$, we see $(1,(-k),(-k)^2,\ldots,(-k)^N)\in V_{N,1-k}$.  To see that this element is also in $V_{N,k}$, set $c_j=-(k-1)^j$.  Then
\begin{eqnarray*}
b_j&=&(k-1)^j+c_j+(-1)^{j+1}\sum_{i=0}^j{j\choose i}c_i\\
&=&(k-1)^j-(k-1)^j+(-1)^j\sum_{i=0}^j{j\choose i}(k-1)^i\\
&=&(-k)^j
\end{eqnarray*}
\end{proof}


\begin{lemma}
\label{bindet}
For positive integers $b$, $n$, let $M_{n,b}$ be the $n\times n$ matrix 
\[
M_{n,b}=\left({b+i\choose j}\right)_{0\leq i,j<n}
\]
Then $\det{M_{n,b}}=1$.
\end{lemma}
\begin{proof}
Define $n\times n$ matrices $L_{n}=({i \choose j})$, $U_{n,b}=({b\choose j-i})$.  $L_{n}$ is unipotent lower-triangular and $U_{n,b}$ is unipotent upper-triangular, so both have determinant 1.  We claim that $M_{n,b}=L_{n}U_{n,b}$, so that $\det{M_{n,b}}=\det{L_n}\det{U_{n,b}}=1$.  Looking at the entry in slot $(i,j)$, this equality reduces to the Vandermonde convolution identity
\[
{b+i\choose j}=\sum_{k=0}^j{i\choose k}{b\choose j-k}
\]

\end{proof}

\begin{lemma}
\label{iso}
For all non-negative integers $n$, $\hat{V}_{2n}$ is a free $\Z$-module of rank $n$.  Additionally, the map $\pi:\hat{V}_{2n}\to\Z^n$ given by
\[
\pi(b_0,\ldots,b_{2n})=(b_{n+1},\ldots,b_{2n})
\]
is an isomorphism
\end{lemma}
\begin{proof}
Here 
$(b_0,\ldots,b_{2n})\in\hat{V}_{2n}$ is determined by the values of $c_j$ for $0\leq j\leq 2n$. We therefore have a surjective map $\varphi_{n}: \Z^{2n+1}\to\hat{V}_{2n}\leq\Z^{2n+1}$, taking $(c_{0},\ldots,c_{2n})$ to $(a_{0},\ldots,a_{2n})$ with
\[
a_{j}=c_{j}+(-1)^{j+1}\sum_{i=0}^{j}{j\choose i}c_{i}
\]

In other words the matrix representing $\varphi_n$ (with respect to the standard basis on $\Z^{2n+1}$) is given by
\[
A_n:=\left(\delta_{i,j}+(-1)^{j+1}{j\choose i}\right)_{0\leq i,j\leq 2n}
\]

If we identify row vectors of length $2n+1$ with the set of polynomials of degree at most $2n$ via the identification $(a_{0},\ldots,a_{2n})\leftrightarrow a_{0}+a_{1}T+\ldots+a_{2n}T^{2n}$, then the $j$-th row of $A_n$ is identified with the polynomial
\[
T^j-(-1-T)^j
\]
This means that the row span of $A_n$ is contained in the set of polynomials $f(T)$ satisfyint $f(T)=-f(-1-T)$.  Such polynomials can be written as $\Q$-linear combinations of $T+\frac{1}{2},(T+\frac{1}{2})^2,\ldots,(T+\frac{1}{2})^{2n-1}$. It follows that rank$(\hat{V}_{2n})=\text{rank}(A_n)\leq n$.

Next let $i:\Z^{n}\to\Z^{2n+1}$, $(x_{0},\ldots,x_{n-1})\mapsto (x_{0},\ldots,x_{n-1},0,0,\ldots,0)$.  We have $\pi\circ\varphi_{n}\circ i(x_{0},\ldots,x_{n-1})=(y_{1},\ldots,y_{n})$, where
\[
y_{j}=(-1)^{n+j}\sum_{i=0}^{n-1}{n+1+j\choose i}x_{i}
\]
By Lemma \ref{bindet}, this map is bijective.  It follows that $\pi$ is surjective.  Since rank$(\hat{V}_{2n})\leq n = \mathrm{rank}(\Z^n)$, we must have that rank$(\hat{V}_{2n})=n$, and $\pi$ is bijective.
\end{proof}

\begin{proof}[Proof of Theorem \ref{MT}]
We need to show that there is a unique element of the form
\[
(b_0,\ldots,b_n,0,0,\ldots,0)\in V_{2n,k}^\Q,
\]
and that $b_0,\ldots,b_n\in\Z$.  It suffices to show there is a unique element of this form in $V_{2n,k}$. Because $V_{2n,k}=(1,(k-1),\ldots,(k-1)^{2n})+\hat{V}_{2n}$, this is equivalent to showing there is a unique element of the form
\[
\underline{a}=(a_0,\ldots,a_n,-(k-1)^{n+1},\ldots,-(k-1)^{2n})\in\hat{V}_{2n}
\]
This is the same as finding an $\underline{a}\in \hat{V}_{2n}$ with $\pi(\underline{a})=(-(k-1)^{n+1},\ldots,-(k-1)^{2n})$.  Such an $\underline{a}$ exists and is unique by Lemma \ref{iso}.

That the values $b_{j,n}(k)$ agree with a polynomial in $k$ will follow from Corollary \ref{recipe} below.
\end{proof}

We summarize the recipe for constructing the coefficients $b_{j,n}(k)$ given in Theorem \ref{MT}. It will follow that these coefficients are interpolated by a polynomial $b_{j,n}(T)$.


\begin{theorem}
\label{recipe}
For fixed integers $0\leq j\leq n$, the coefficients $b_{j,n}(k)$ given in Theorem \ref{MT} are values of a polynomial $b_{j,n}(T)$ at $T=k$, which is of degree at most $2n$. This polynomial can be computed explicitly as follows.

\noindent Let $M_{n}$ be the $n\times n$ matrix 
\[
M_n=\left[(-1)^{n+i}{n+1+i\choose j}\right]_{0\leq i,j\leq n-1}
\]
Let $D_{n}$ be the $(n+1)\times n$ matrix
\[
D_n=\left[(-1)^{i+1}{i\choose j}+\delta_{i,j}\right]_{\substack{0\leq i\leq n\\0\leq j\leq n-1}}
\]
where $\delta_{i,j}$ is the Kronecker delta.  Then $M_{n}$ is invertible over the integers, and we have the matrix equation
\begin{equation}
\label{eq1}
\left(  \begin{array}{c}
b_{0,n}(k)\\
b_{1,n}(k)\\
\vdots\\
b_{n,n}(k)\\
\end{array} \right)
=
\left(  \begin{array}{c}
(k-1)^{0}\\
(k-1)^{1}\\
\vdots\\
(k-1)^{n}\\
\end{array} \right)
-
D_{n}\cdot M_{n}^{-1}\cdot
\left(  \begin{array}{c}
(k-1)^{n+1}\\
(k-1)^{n+2}\\
\vdots\\
(k-1)^{2n}\\
\end{array} \right)
\end{equation}
In particular this shows that $b_{j,n}(k)$ is given by a polynomial in $k$, of degree at most $2n$, having integer coefficients.
\end{theorem}

\begin{proof}
The formula \eqref{eq1} follows from the proof of Theorem \ref{MT}. This formula implies that $b_{j,k}(k)$ is given by a polynomial in $k$, of degree at most $2n$, having integer coefficients.
\end{proof}

\begin{definition}
{\em For integers $j$, $n$, $k$, with $j,n\geq0$, we let $b_{j,n}(k)$ denote the coefficients arising from Theorem \ref{MT}.  We call these \emph{\optimal coefficients}.  We also denote by $b_{j,n}(T)$ the polynomial giving these coefficients, and call these \emph{\optimal polynomials}.  By convention, we take $b_{j,n}(T)=0$ for $n+1\leq j\leq 2n$, and we say that $b_{j,n}(T)$ is not defined for $j\geq 2n+1$.
}
\end{definition}

Theorem \ref{recipe} provides
  a recipe for computing the \optimal coefficients in terms of matrices involving binomial coefficients. In this sense they are like the Bernoulli numbers, which can be computed by a similar expression.

Proposition \ref{swap} says that $V_{2n,k}=V_{2n,1-k}$.  We may therefore make the substitution 
$k\leftrightarrow 1-k$ in Theorem \ref{recipe} to get the following:

\begin{corollary}
\label{arith}
Let $k$, $n$ be given, and set $N=2n$.  Let $D_{n}$, $M_{n}$ be as in the statement of Theorem \ref{recipe}.  Then

\begin{equation}
\label{eq2}
\left(  \begin{array}{c}
b_{0}(k,n)\\
b_{1}(k,n)\\
\vdots\\
b_{n}(k,n)\\
\end{array} \right)
=
\left(  \begin{array}{c}
(-k)^0\\
(-k)^1\\
\vdots\\
(-k)^n\\
\end{array} \right)
-
D_{n}\cdot M_{n}^{-1}\cdot
\left(  \begin{array}{c}
(-k)^{n+1}\\
(-k)^{n+2}\\
\vdots\\
(-k)^{2n}\\
\end{array} \right)
\end{equation}

\end{corollary}

We can combine Theorem \ref{recipe} and Corollary \ref{arith} to obtain the following characterization of the 
\optimal polynomials $b_{j,n}(T)$:

\begin{proposition}
\label{congforb}
Fix integers $j$, $n$, with $j\leq 2n$.  The \optimal polynomial $b_{j,n}(T)\in\Z[T]$ is the unique polynomial
of degree at most $2n$  satisfying the following conditions:
\begin{enumerate}[(i)]
\item $b_{j,n}(T)\equiv (T-1)^j\mod{(T-1)^{n+1}}$
\item $b_{j,n}(T)\equiv (-T)^j\mod{T^{n+1}}$
\end{enumerate}
\end{proposition}
\begin{proof}
For $n+1\leq j\leq 2n$, $b_{j,n}(T)=0$, and the result can be seen directly.  For $j\leq n$, Corollary \ref{recipe} shows that $\deg(b_{j,n}(T))\leq 2n$, and that $b_{j,n}(T)$ is equal to $(T-1)^j$ plus a $\Z$-linear combination of $(T-1)^{n+1},\ldots,(T-1)^{2n}$.  This shows that \emph{(i)}, \emph{(ii)} are satisfied.  Condition \emph{(iii)} similarly follows from Corollary \ref{arith}.

For uniqueness, the Chinese remainder theorem says that conditions \emph{(ii)}, \emph{(iii)} determine the residue class of $b_{j,n}(T)$ modulo $T^{n+1}(T-1)^{n+1}$.  There will, a fortiori, be only one polynomial with rational coefficients and degree at most $2n+1$ in this residue class. Our computations have shows that this polynomial in fact has integer coefficients and degree at most $2n$.
\end{proof}

Theorem \ref{th-main}, stated in the introduction, now follows from the combination of Theorem \ref{MT}, Theorem \ref{recipe}, and Proposition \ref{congforb}.


\section{Exceptional Congruences and Bernoulli Numbers}

We now  investigate the situations under which the congruences \eqref{goodp}, \eqref{badp}
 hold modulo some larger power of $p$ than given by Theorem \ref{MT}.  We term these {\em exceptional congruences}.
 In the case of Wolstenholme's theorem, we have that the exceptional congruence
\[
{2p-1\choose p-1}\equiv 1\pmod{p^4}
\]
holds if and only if $p$ divides the numerator of the Bernoulli number $B_{p-3}$ (this follows from the results of van Hamme \cite{vH00} and Glaisher \cite{Gla1900a}).  We establish a similar result, which shows that the congruences \eqref{goodp}, \eqref{badp} hold modulo an extra power of $p$ if and only if either $p|B_{p-2n-3}$, or $p|k(k-1)$.

Let non-negative integers $n$, $k$ be given, and choose $c_0,c_1,\ldots\in\Z$ so that the generalized Wolstenholme congruence associated with the data $[k,(c_0,c_1,\ldots),2n]$ is the optimal one, given by Theorem \ref{MT} (the $c_i$ are not uniquely determined by this condition).  Let $b_0,b_1,\ldots,b_{2n+2}$ be given by \eqref{defb}, so that $b_j=b_{j,n}(k)$ for $j=0,1,\ldots,n$, and $b_j=0$ for $j=n+1,n+2,\ldots,2n$. The values of $b_{2n+1}$ and $b_{2n+3}$ will depend on the choice of the $c_i$.

\begin{proposition}
Define
\[
C_n(k) :=(n+1)b_{2n+1}+b_{2n+2}
\]
Then, independent of the choice of $c_i$ giving the optimal congruence, we have $C_n(k)=k^{n+1}(k-1)^{n+1}$.
\end{proposition}
\begin{proof}
First we show that the value of $C_n(k)$ depends only on $n$ and $k$, but not the choice of $c_i$. Let $p:V_{2n+2,k}\to V_{2n,k}$, $(b_0,\ldots,b_{2n+2})\mapsto (b_0,\ldots,b_{2n})$ be the projection map (where $V_{\cdot,k}$ is given by Definition \ref{defV}).  From the construction of the spaces $V_{2n+1,k}$, $V_{2n,k}$, $p$ is surjective. Lemma \ref{iso} says that rank$(V_{2n+2,k})=n+1$, rank$(V_{2n,k})=n$.  If follows that $U:=p^{-1}(b_0,b_1,\ldots,b_n,0,\ldots,0)$ is a $\Z$ torsor.  Therefore, if we can exhibit $(b'_0,\ldots,b'_{2n+2})\neq (b_0,\ldots,b_{2n+3})\in U$ such that $(n+1)b'_{2n+1}+b'_{2n+2}=(n+1)b_{2n+1}+b_{2n+2}$, we will be done.

If we take $c'_{i}=c_{i}$ for $i\neq 2n+1$, and $c'_{2n+1}=c_{2n+1}+1$, we will have that $b'_i=b_i$ for $i\leq 2n$, $b'_{2n+1}=b_{2n+1}+{2n+1\choose 1}$, and $b'_{2n+2}=b_{2n+2}-{2n+2\choose 2}$. It follows directly that $(n+1)b'_{2n+1}+b'_{2n+2}=(n+1)b_{2n+1}+b_{2n+2}$.

Next we show that $C_n(k)$ agrees with a polynomial in $k$.  Using the same process as in Corollary \ref{recipe}, we may solve for the data $c_0,\ldots,c_{2n}$ to give the \optimal congruence.  
We will use this data (with $c_i=0$ for $i\geq 2n+1$).  We can then compute $b_{2n+1}$, $b_{2n+2}$ in the following way:

\noindent Let $M_{n}$ be the $n\times n$ matrix 
\[
M_n=\left[(-1)^{n+i}{n+1+i\choose j}\right]_{0\leq i,j\leq n-1}
\]
Let $A_{n}$ be the $2\times n$ matrix
\[
A_n=\left[(-1)^{i+1}{i\choose j}+\delta_{i,j}\right]_{\substack{2n+1\leq i\leq 2n+2\\0\leq j\leq n-1}}
\]
where $\delta_{i,j}$ is the Kronecker delta.  Then $M_{n}$ is invertible over the integers, and we have the matrix equation
\begin{equation}
\label{eqextra}
\left(  \begin{array}{c}
b_{2n+1}\\
b_{2n+2}\\
\end{array} \right)
=
\left(  \begin{array}{c}
(k-1)^{2n+1}\\
(k-1)^{2n+2}\\
\end{array} \right)
-
A_{n}\cdot M_{n}^{-1}\cdot
\left(  \begin{array}{c}
(k-1)^{n+1}\\
(k-1)^{n+2}\\
\vdots\\
(k-1)^{2n}\\
\end{array} \right)
\end{equation}

This shows that $b_{2n+1}$, $b_{2n+2}$ are polynomials in $k$.  Moreover, $b_{2n+1}$ is equal to $(k-1)^{2n+1}$ plus a $\Z$-linear combination of $(k-1)^{n+1},\ldots,(k-1)^{2n}$, so that $b_{2n+1}$ is monic in $k$, of degree $2n+1$, and $(k-1)^{n+1}| b_{2n+1}$.  Similarly, $b_{2n+2}$ is monic of degree $2n+2$, and $(k-1)^{n+1}|b_{2n+2}$.  It follows that $C_n(k)=(n+1)b_{2n+1}+b_{2n+2}$ is monic of degree $2n+2$, with $(k-1)^{2n+1}|C_n(k)$.  Moreover, $C_n(k)$ is determined by the set $V_{2n+2,k}$, and Lemma \ref{swap} says that $V_{2n+2,k}=V_{2n+2,1-k}$.  We may therefore make the substitution $k\leftrightarrow 1-k$ to get the $k^{n+1}|C_n(k)$.  By the Chinese remainder theorem, $k^{n+1}(k-1)^{n+1}|C_n(k)$. The only monic polynomial of degree $2n+2$ which is divisible by $k^{n+1}(k-1)^{n+1}$ is $k^{n+1}(k-1)^{n+1}$, so we conclude $C_n(k)=k^{n+1}(k-1)^{n+1}$.
\end{proof}

We now consider the possibility of extra powers of $p$ in the congruences \eqref{goodp}, \eqref{badp}.
For all integers $k$, $n$ with $n\geq0$, and all odd primes $p$, define
\[
E(k,n,p):={kp-1\choose p-1}-\sum_{j=0}^n b_{j,n}(k)p^j\H(\{1\}^j)
\]

\begin{proposition}
Suppose $p\geq 2n+5$.  Then
\[
E(k,n,p)\equiv \frac{-B_{p-3-2n}}{2n+3}k^{n+1}(k-1)^{n+1}p^{2n+3}\pmod{p^{2n+4}}
\]
\end{proposition}
\begin{proof}
By Theorem \ref{gwc}, we have
\[
E(k,n,p)\equiv b_{2n+1}p^{2n+1}\H(\{1\}^{2n+1})+b_{2n+2}p^{2n+2}\H(\{1\}^{2n+2})\hspace{-.2cm}\pmod{p^{2n+4}}
\]
Using Proposition \ref{scc}, we can write
\begin{eqnarray*}
E(k,n,p)&\equiv& -b_{2n+1}p^{2n+3}\frac{2n+2}{2(2n+3)}B_{p-3-2n}-b_{2n+2}p^{2n+3}\frac{B_{p-3-2n}}{2n+3}\\
&\equiv&\frac{-B_{p-3-2n}}{n+3}\left((n+1)b_{2n+1}-b_{2n+2}\right)\\
&\equiv&\frac{-B_{p-3-2n}}{n+3}k^{n+1}(k-1)^{n+1}\pmod{p^{2n+4}}
\end{eqnarray*}
\end{proof}

\begin{proposition}
Suppose $p=2n+3$.  Then
\[
E(k,n,p)\equiv -k^{n+1}(k-1)^{n+1}p^{2n+2}\pmod{p^{2n+3}}
\]
\end{proposition}
\begin{proof}
By Theorem \ref{gwc}, we have
\[
E(k,n,p)= b_{2n+1}p^{2n+1}\H(\{1\}^{2n+1})+b_{2n+2}p^{2n+2}\H(\{1\}^{2n+2})
\]
Using Proposition \ref{scc}, we can write
\begin{eqnarray*}
E(k,n,p)&\equiv&\frac{b_{2n+1}}{2}p^{2n+2}-b_{2n+2}p^{2n+2}\\
&\equiv&\left(-\frac{(p-1)b_{2n+1}}{2}-b_{2n+1}\right)p^{2n+2}\\
&\equiv&-\left((n+1)b_{2n+1}+b_{2n+2}\right)\\
&\equiv&-k^{n+1}(k-1)^{n+1}p^{2n+2}\pmod{p^{2n+3}}
\end{eqnarray*}
\end{proof}

We can now state the precise conditions under which the congruences \eqref{goodp}, \eqref{badp} hold modulo a larger power of $p$ than is given by Theorem \ref{MT}.  The following Theorem is an immediate consequence of the preceding two propositions:

\begin{theorem}
Let $n\geq0$, $k$ be integers, $p$ an odd prime, and $b_{j,n}(T)$ the \optimal polynomials (characterized by Proposition \ref{congforb}).
\begin{enumerate}[(i)]
\item Suppose $p\geq 2n+5$.  The congruence
\[
{kp-1\choose p-1}\equiv \sum_{j=0}^nb_{j,n}(k)p^j\H(\{1\}^j; p-1)\pmod{p^{2n+4}}
\]
holds if and only if either $p|B_{p-3-2n}$ or $k\equiv 0,1\pmod{p}$.

\item  Suppose $p=2n+3$.  The congruence
\[
{kp-1\choose p-1}\equiv \sum_{j=0}^nb_{j,n}(k)p^j\H(\{1\}^j; p-1)\pmod{p^{2n+3}}
\]
holds if and only if $k\equiv 0,1\pmod{p}$.
\end{enumerate}
\end{theorem}

%
%
\section{Properties of the \Optimal Polynomials}\label{sec5}

The \optimal polynomials  $b_{j,n}(T)$ satisfy many arithmetic and congruence relations. 

\begin{proposition}
\label{propbasic}
The \optimal polynomials $b_{j,n}(T)$ satisfy the following properties:
\begin{enumerate}[(i)]
\item For all non-negative integers $n$, $b_{0,n}(T)=1$
\item For all non-negative integers $n$, $b_{n,n}(T)=T^n(T-1)^n$
\item For all non-negative integers $j\leq2n$, $T^j(T-1)^j$ divides $b_{j,n}(T)$.
\end{enumerate}
\end{proposition}
\begin{proof}
\emph{(i)} and \emph{(ii)} follow by checking that the given polynomial satisfies the conditions of Proposition \ref{congforb}.  \emph{(iii)} also follows immediately from Proposition \ref{congforb}.
\end{proof}

For fixed $j$, the polynomials $b_{j,n}(T)$ depend on $n$. One exception to this is that $b_{0,n}(T)=1$ for all non-negative integers $n$.  Examining the table in Section 1.2, we see that $b_{1,n}(T)$ does indeed depend on $n$.  However, $b_{1,n}(T)+b_{2,n}(T)=T^2-T$ for all $n$.

This is the first in a family of equations giving linear combinations of the \optimal polynomials $b_{j,n}(T)$ ($j$ varying) which are independent of $n$.

\begin{proposition}
Let $m$ be a non-negative integer.  Suppose we are given $f(T)=\sum_{j=0}^ma_jT^j\in\Q[T]$ satisfying $f(T)=f(-1-T)$.  Then, for all non-negative integers $n$ with $2n\geq m$, we have
\[
\sum_{j=0}^ma_jb_{j,n}(T)=f(T-1)
\]
\end{proposition}
\begin{proof}
Define
\[
g(T)=\sum_{j=0}^ma_jb_{j,n}(T)
\]
By Proposition \ref{congforb}, we have $b_{j,n}(T)\equiv (T-1)^j\pmod{(T-1)^{n+1}}$, so that
\begin{eqnarray*}
g(T)&\equiv& \sum_{j=0}^m a_j (T-1)^j\\
&\equiv&f(T-1)\pmod{(T-1)^{n+1}}
\end{eqnarray*}
Similarly, by Proposition \ref{congforb}, we have $b_{j,n}(T)\equiv (-T)^j\pmod{T^{n+1}}$, so that
\begin{eqnarray*}
g(T)&\equiv& \sum_{j=0}^m a_j (-T)^j\\
&\equiv&f(-T)\\
&\equiv&f(T-1)\pmod{(T-1)^{n+1}}
\end{eqnarray*}
By the Chinese remainder theorem, it follows that\\
$g(T)\equiv f(T-1)\pmod{T^{n+1}(T-1)^{n+1}}$.  Now, we just note that $g(T)$ and $f(T-1)$ have degree at most $2n+1$, so they must be equal.
\end{proof}

\paragraph{\bf Acknowledgments.} I thank J.\ Lagarias for pointing out this problem, 
and for  editorial comments and references.  I also thank E.\ H.\ Brooks for introducing the author to Wolstenholme's theorem, and for helpful comments. This work was supported in part by NSF grants  DMS-0943832 and DMS-1101373.

\bibliographystyle{amsplain}

\end{document}